\documentclass[11pt,leqno]{article}
\usepackage{amsmath, amscd, amsthm, amssymb, graphics, xypic, mathrsfs, setspace, fancyhdr, times, bm, enumitem}
\usepackage[letterpaper,top=1in, bottom=1in, left=1in, right=1in]{geometry}
\usepackage[colorlinks=true,pagebackref=true]{hyperref} 
\hypersetup{backref}

\newcommand{\colim}{\operatorname{colim}}
\newcommand{\Spec}{\operatorname{Spec}}

\newcommand{\isomt}{{\stackrel{{\scriptscriptstyle{\sim}}}{\;\rightarrow\;}}}

\newcommand{\sma}{{\scriptstyle{\wedge}}}

\newcommand{\Singaone}{\operatorname{Sing}^{\aone}\!\!}

\renewcommand{\O}{{\mathcal O}}
\renewcommand{\hom}{\operatorname{Hom}}

\newcommand{\Q}{{\mathbb Q}}
\newcommand{\Z}{{\mathbb Z}}

\newcommand{\aone}{{\mathbb A}^1}

\newcommand{\gm}[1]{{{\mathbf G}_{m}^{#1}}}

\newcommand{\et}{\text{\'et}}
\newcommand{\ho}[1]{\mathscr{H}({#1})}

\newcommand{\Nis}{\operatorname{Nis}}
\newcommand{\Zar}{\operatorname{Zar}} 

\newcommand{\Frac}{\mathrm{Frac}}

\newcommand{\Sm}{\mathrm{Sm}}

\newcommand{\op}[1]{\operatorname{#1}}
\newcommand{\fppf}{\mathrm{fppf}}

\newcommand{\Addresses}{{
  \bigskip
  \footnotesize

  A.~Asok, \textsc{Department of Mathematics, University of Southern California,
    Los Angeles, CA 90089-2532, United States;} \textit{E-mail address:} \url{asok@usc.edu}

  \medskip

  M.~Hoyois, \textsc{Department of Mathematics, University of Southern California,
    Los Angeles, CA 90089-2532, United States;} \textit{E-mail address:} \url{hoyois@usc.edu}

  \medskip

  M.~Wendt, \textit{E-mail address:} \url{m.wendt.c@gmail.com}

}}

\newcounter{intro}
\theoremstyle{plain}
\newtheorem{thm}{Theorem}[section]

\newtheorem{lem}[thm]{Lemma}

\newtheorem{prop}[thm]{Proposition}
\newtheorem*{claim*}{Claim}  

\newtheorem{question}[thm]{Question}

\newtheorem*{thm*}{Theorem}
\newtheorem*{problem*}{Problem}

\newtheorem{thmintro}{Theorem}

\newtheorem{defnintro}[thmintro]{Definition}

\theoremstyle{definition}

\theoremstyle{remark}
\newtheorem{rem}[thm]{Remark}
\newtheorem{remintro}[thmintro]{Remark}

\newtheorem{ex}[thm]{Example}

\numberwithin{equation}{section}

\begin{document}
\pagestyle{fancy}
\renewcommand{\sectionmark}[1]{\markright{\thesection\ #1}}
\fancyhead{}
\fancyhead[LO,R]{\bfseries\footnotesize\thepage}
\fancyhead[LE]{\bfseries\footnotesize\rightmark}
\fancyhead[RO]{\bfseries\footnotesize\rightmark}
\chead[]{}
\cfoot[]{}
\setlength{\headheight}{1cm}

\author{Aravind Asok\thanks{Aravind Asok was partially supported by National Science Foundation Award DMS-1254892.} \and Marc Hoyois\thanks{Marc Hoyois was partially supported by National Science Foundation Award DMS-1761718.}\and Matthias Wendt}

\title{{\bf Affine representability results in $\aone$-homotopy theory\\ III: finite fields and complements}}
\date{}
\maketitle

\begin{abstract}
We give a streamlined proof of $\aone$-representability for $G$-torsors under ``isotropic" reductive groups, extending previous results in this sequence of papers to finite fields.  We then analyze a collection of group homomorphisms that yield fiber sequences in ${\mathbb A}^1$-homotopy theory, and identify the final examples of motivic spheres that arise as homogeneous spaces for reductive groups.
\end{abstract}


\section{Introduction/Statement of Results}
Suppose $k$ is a field.  We study torsors under algebraic groups considered in the following definition.

\begin{defnintro}
\label{defnintro:isotropic}
If $\op{G}$ is a reductive algebraic $k$-group scheme, we will say that $\op{G}$ is ``isotropic" if each of the almost $k$-simple components of the derived group of $\op{G}$ contains a $k$-subgroup scheme isomorphic to $\gm{}$.
\end{defnintro}

Write $\ho{k}$ for the (unstable) Morel--Voevodsky $\aone$-homotopy category \cite{MV}.  Write $B\op{G}$ for the usual bar construction of $\op{G}$ (which can be thought of as a simplicial presheaf on the category of smooth $k$-schemes).  If $X$ is a smooth $k$-scheme, then write $[X,B\op{G}]_{\aone}$ for the set $\hom_{\ho{k}}(X,B\op{G})$.  The main goal of this paper is to establish the following representability result about Nisnevich locally trivial $\op{G}$-torsors.

\begin{thmintro}
\label{thmintro:isotropictorsors}
Suppose $k$ is a field, and $\op{G}$ is an ``isotropic" reductive $k$-group.  For every smooth affine $k$-scheme $X$, there is a bijection
\[
\op{H}^1_{\Nis}(X,\op{G}) \cong [X,B\op{G}]_{\aone}
\]
that is functorial in $X$.
\end{thmintro}

In \cite[Theorem 4.1.3]{AHWII}, Theorem~\ref{thmintro:isotropictorsors} was proved under the more restrictive assumption that $k$ is infinite.  By \cite[Theorem 2.3.5]{AHWII}, in order to establish Theorem~\ref{thmintro:isotropictorsors}, it suffices to prove that the functor $X \mapsto \op{H}^1_{\Nis}(X,\op{G})$ is $\aone$-invariant on smooth affine schemes, i.e., if for every smooth affine $k$-scheme $X$, the pullback along the projection $X \times \aone \to X$ induces a bijection $\op{H}^1_{\Nis}(X,\op{G}) \isomt \op{H}^1_{\Nis}(X \times \aone,\op{G})$.

Using a recent refinement of the Gabber presentation lemma over finite fields first stated by F. Morel \cite[Lemma 1.15]{MField} (where it is attributed to Gabber) and proven by A. Hogadi and G. Kulkarni \cite{HK}, we establish affine homotopy invariance over finite fields in Theorem~\ref{thm:affinehomotopyinvarianceforGtorsors}.

\begin{remintro}
Over a finite field, one knows that all reductive $k$-group schemes are quasi-split by a result of Lang, cf. \cite{Lang}.  In particular, semi-simple group schemes will automatically be isotropic in this case.
\end{remintro}

As immediate consequences, we may remove the assumption that $k$ is infinite in many of the results stated in \cite{AHWII}.  In particular, we establish the following result.

\begin{thmintro}
\label{thmintro:homogeneousspaces}
Assume $k$ is a field.  If $\op{H} \to \op{G}$ is a closed immersion of ``isotropic" reductive $k$-group schemes, and the $\op{H}$-torsor $\op{G} \to \op{G}/\op{H}$ is Nisnevich locally split, then for any smooth affine $k$-scheme $X$, there is a bijection
\[
\pi_0(\Singaone \op{G}/\op{H})(X) \cong [X,\op{G}/\op{H}]_{\aone}.
\]
\end{thmintro}

Theorem~\ref{thm:generalizedflagsareaonenaive} contains a similar result for certain generalized flag varieties under ``isotropic" reductive $k$-group schemes and the remainder of the main results (e.g., Theorem~\ref{thm:evenspheres}) contain some useful explicit examples.

\subsubsection*{Acknowledgements}
We would like to thank A. Hogadi and G. Kulkarni for sharing an early draft of \cite{HK}.

\subsubsection*{Notation/Preliminaries}
Throughout the paper, $k$ will be a field.  Following \cite{AHW,AHWII}, we use the following terminology:
\begin{itemize}[noitemsep,topsep=1pt]
\item $\Sm_k$ is the category of smooth $k$-schemes;
\item $\mathrm{sPre}(\Sm_k)$ is the category of simplicial presheaves on $\Sm_k$; objects of this category will typically be denoted by script letters $\mathscr{X},\mathscr{Y}$, etc.;
\item if $t$ is a topology on $\Sm_k$, we write $R_{t}$ for the fibrant replacement functor for the injective $t$-local model structure on $\mathrm{sPre}(\Sm_k)$ (see \cite[Section 3.1]{AHW});
\item $\Singaone{}$ is the singular construction (see \cite[Section 4.1]{AHW});
\item $\ho{k}$ is the Morel--Voevodsky unstable $\aone$-homotopy category (see \cite[Section 5]{AHW});
\item if $\mathscr{X}$ and $\mathscr{Y}$ are simplicial presheaves on $\Sm_k$, we write $[\mathscr{X},\mathscr{Y}]_{\aone} := \hom_{\ho{k}}(\mathscr{X},\mathscr{Y})$.
\end{itemize}
Throughout the text, we will speak of reductive group schemes; following SGA3 \cite{SGA3-3}, by convention such group schemes have geometrically connected fibers.

\section{Proofs}
\subsection{Homotopy invariance revisited}
In \cite[Proposition 3.3.4]{AHWII}, we developed a formalism for establishing affine homotopy invariance of certain functors; this method was basically an extension of a formalism developed by Colliot-Th\'el\`ene and Ojanguren \cite[Th{\'e}or{\`e}me 1.1]{CTO} and relied on a refined Noether normalization result (a ``presentation lemma") that held over infinite fields \cite[Lemma 1.2]{CTO}.  In Theorem \ref{thm:gabberfinite}, we recall a version of a stronger ``presentation lemma" due initially to Gabber.  Then, in Proposition~\ref{prop:formalism}, we simplify and generalize \cite[Proposition 3.3.4]{AHWII}.

\subsubsection*{Gabber's lemma}
The following result was initially stated in \cite[Lemma 1.15]{MField} where it was attributed to private communication with Gabber.  In the case $k$ is infinite, a detailed proof of a more general result is given in \cite[Theorem 3.1.1]{CTHK}, while when $k$ is finite the result is established recently by Hogadi and Kulkarni \cite[Theorem 1.1]{HK}.  In fact, in what follows we will not need the full strength of this result.

\begin{thm}
\label{thm:gabberfinite}
Suppose $F$ is a field, and suppose $X$ is a smooth affine $F$-variety of dimension $d \geq 1$.  Let $Z \subset X$ be a principal divisor defined by an element $f \in \O_X(X)$ and $p \in Z$ a closed point.  There exist i) a Zariski open neighborhood $U$ of $p$ in $X$, ii) a morphism $\Phi: U \to {\mathbb A}^d_F$, iii) an open neighborhood $V \subset {\mathbb A}^{d-1}_F$ of the composite $\Psi$
\[
U \stackrel{\Phi}{\longrightarrow} {\mathbb A}^d_F \stackrel{\pi}{\longrightarrow} {\mathbb A}^{d-1}_F
\]
(where $\pi$ is the projection onto the first $d-1$ coordinates) such that
\begin{enumerate}[noitemsep,topsep=1pt]
\item the morphism $\Phi$ is \'etale;
\item setting $Z_V := Z \cap \Psi^{-1}V$, the morphism $\Psi|_{Z_V}: Z_V \to V$ is finite;
\item the morphism $\Phi|_{Z_V}: Z_V \to \aone_V = \pi^{-1}(V)$ is a closed immersion;
\item there is an equality $Z_V = \Phi^{-1}\Phi(Z_V)$.
\end{enumerate}
In particular, the morphisms $\Phi$ and $j: \aone_V \setminus Z_V \to \aone_V$ yield a Nisnevich distinguished square of the form
\[
\xymatrix{
U \setminus Z_V \ar[r]\ar[d] & U \ar[d] \\
\aone_V \setminus \Phi(Z_V)  \ar[r] & \aone_V.
}
\]
\end{thm}

\subsubsection*{A formalism for homotopy invariance}
The following result simplifies and generalizes \cite[Proposition 3.3.4]{AHWII}.

\begin{prop}
\label{prop:formalism}
Suppose $k$ is a field. Let $\mathbf{F}$ be a presheaf of pointed sets on the category $\mathbf C$ of essentially smooth affine $k$-schemes with the following properties:
\begin{enumerate}[noitemsep,topsep=1pt]
\item If $\Spec A\in \mathbf C$ and $S\subset A$ is a multiplicative subset, the canonical map $\colim_{f\in S}\mathbf{F}(A_f) \to \mathbf{F}(S^{-1}A)$ has trivial kernel.
\item For every finitely generated separable field extension $L/k$ and every integer $n \geq 0$, the restriction map
\[
\mathbf{F}(L[t_1,\ldots,t_n]) \longrightarrow \mathbf{F}(L(t_1,\ldots,t_n))
\]
has trivial kernel.
\item For every Nisnevich square
\[
	\xymatrix{
	W \ar@{^{(}->}[r]\ar[d] & V \ar[d]\\
	U \ar@{^{(}->}[r] & X,
	}
\]
in $\mathbf C$ where $W\subset V$ is the complement of a principal divisor, the map
    \[
    \ker(\mathbf{F}(X) \to \mathbf{F}(U)) \longrightarrow \ker(\mathbf{F}(V) \to \mathbf{F}(W))
    \]
    is surjective.
\end{enumerate}
If $\Spec B\in\mathbf{C}$ is local, then, for any integer $n \geq 0$, the restriction map
\[
\mathbf{F}(B[t_1,\ldots,t_n]) \longrightarrow \mathbf{F}(\Frac(B)(t_1,\ldots,t_n))
\]
has trivial kernel.
\end{prop}

\begin{proof}
We proceed by induction on the dimension $d$ of $B$.  The case $d = 0$ is immediate from (2).  Assume we know the result in dimension $\leq d-1$. Suppose $\xi \in \ker(\mathbf{F}(B[t_1,\ldots,t_n]) \to \mathbf{F}(\Frac(B)(t_1,\ldots,t_n)))$. By (2), the image of $\xi$ in $\mathbf{F}(\Frac(B)[t_1,\ldots,t_n])$ is trivial. By (1), we conclude that there is an element $g \in B \setminus 0$ such that $\xi$ restricts to the trivial element in $\mathbf{F}(B_g[t_1,\ldots,t_n])$.

By Theorem~\ref{thm:gabberfinite} applied to $X=\Spec B$, $Z$ the principal divisor defined by $g$, and $p$ the closed point in $\Spec B$, we may find a Nisnevich square
\[
	\xymatrix{
	\Spec B_g \ar@{^{(}->}[r]\ar[d] & \Spec B \ar[d]\\
	U \ar@{^{(}->}[r] & \Spec A[x]
	}
\]
with $A$ an essentially smooth local ring of dimension $d-1$. Note that the open immersion $U\subset \Spec A[x]$ is affine, since it is so after the surjective \'etale base change $U\amalg \Spec B \to \Spec A[x]$. Now, by (3), since $\xi$ lies in the kernel of $\mathbf{F}(B[t_1,\ldots,t_n]) \to \mathbf{F}(B_g[t_1,\ldots,t_n])$, we may find
\[
\xi' \in \ker(\mathbf{F}(A[x][t_1,\ldots,t_n]) \to \mathbf{F}(U[t_1,\ldots,t_n]))
\]
lifting $\xi$. In particular, the image of the class $\xi'$ in $\mathbf{F}(\Frac(A)(x,t_1,\ldots,t_n))$ must also be trivial. However, $A[x][t_1,\ldots,t_n] = A[x,t_1,\ldots,t_n]$ and since $A$ has dimension $d-1$, we conclude that $\xi'$ is trivial, which means that $\xi$ must also be trivial and we are done.
\end{proof}

\begin{rem}
The proof of Proposition~\ref{prop:formalism} only uses assertions 1, 3, and 4 of Theorem~\ref{thm:gabberfinite}, and it may be possible to give a shorter and more self-contained proof of these assertions.
\end{rem}

\subsubsection*{Homotopy invariance for $G$-torsors over arbitrary fields}
We now apply Proposition~\ref{prop:formalism} in the case of the functor ``isomorphism classes of Nisnevich locally trivial $\op{G}$-torsors" under an ``isotropic" reductive $k$-group $\op{G}$ (see Definition~\ref{defnintro:isotropic}).

\begin{thm}
\label{thm:affinehomotopyinvarianceforGtorsors}
If $k$ is a field, and $\op{G}$ is an ``isotropic" reductive $k$-group scheme, then, for any smooth $k$-algebra $A$ and any integer $n \geq 0$, the map
\[
\op{H}^1_{\Nis}(\Spec A,\op{G}) \longrightarrow \op{H}^1_{\Nis}(\Spec A[t_1,\ldots,t_n],\op{G})
\]
is a pointed bijection.
\end{thm}

\begin{proof}
Repeat the proof of \cite[Theorem 3.3.7]{AHWII}, replacing appeals to \cite[Proposition 3.3.4]{AHWII} with reference to Proposition~\ref{prop:formalism}.  As the formulation of Proposition~\ref{prop:formalism} differs slightly from that of \cite[Proposition 3.3.4]{AHWII}, we include the argument here.

We want to show that every Nisnevich locally trivial $\op{G}$-torsor $\mathscr{P}$ over the ring $A[t_1,\ldots,t_n]$ is extended from $A$.  After \cite[Corollary 3.2.6]{AHWII}, which is a local-to-global principle for torsors under a reductive group scheme, it suffices to show that for every maximal ideal ${\mathfrak m}$ of $A$, the $G$-torsor $\mathscr{P}_{{\mathfrak m}}$ over $A_{{\mathfrak m}}[t_1,\ldots,t_n]$ is extended from $A_{{\mathfrak m}}$.  In fact, we will show that $\mathscr{P}_{{\mathfrak m}}$ is a trivial torsor.

We claim that the functor from $k$-algebras to pointed sets given by $A \mapsto \op{H}^1_{\Nis}(\Spec A,\op{G})$ satisfies the axioms of Proposition~\ref{prop:formalism}.  The first point is an immediate consequence of the fact that $\op{G}$ has finite presentation by appeal to \cite[Lemma 2.3.3]{AHWII}.  Recall from \cite[Definition 2.3.1]{AHWII} that we write $B\mathbf{Tors}_{\Nis}(\op{G})$ for the simplicial presheaf whose value on a smooth scheme $U$ is the nerve of the groupoid of $\op{G}$-torsors over $U$.  The third point is then a formal consequence of the fact that the functor $\op{H}^1_{\Nis}(-,\op{G})$ can be identified with the set of connected components $\pi_0(B\mathbf{Tors}_{\Nis}(\op{G}))$ since $B\mathbf{Tors}_{\Nis}(\op{G})$ satisfies Nisnevich excision essentially by definition (see \cite[\S 2.3]{AHWII} for more details).  Finally, the second point follows by appeal to results of Raghunathan \cite{Raghunathan1,Raghunathan2}, which are conveniently summarized in \cite[Proposition 2.4 and Th\'eor\`eme 2.5]{CTO}; this is where the assumption that $\op{G}$ is ``isotropic'' is used.

The hypotheses of Proposition~\ref{prop:formalism} having been satisfied, to conclude that $\mathscr{P}_m$ is trivial, it suffices to show that it becomes trivial over the field $\operatorname{Frac}(A_{\mathfrak m})(t_1,\ldots,t_n)$, but this follows immediately from the fact that a field has no nontrivial Nisnevich covering sieves.
\end{proof}

\subsubsection*{Representability results}
Granted Theorem~\ref{thm:affinehomotopyinvarianceforGtorsors}, we can immediately generalize a number of results from \cite{AHWII}.  For ease of reference, we restate the relevant results here.  We begin by establishing Theorem~\ref{thmintro:isotropictorsors} from the introduction.

If $\mathscr{F}$ is a simplicial presheaf on $\Sm_k$, and $\tilde{\mathscr{F}}$ is a Nisnevich-local and $\aone$-invariant fibrant replacement of $\mathscr{F}$, then there is a canonical map $\Singaone \mathscr{F} \to \tilde{\mathscr{F}}$ that is well-defined up to simplicial homotopy.  Recall from \cite[Definition 2.1.1]{AHWII} that a simplicial presheaf $\mathscr{F}$ on $\Sm_k$ is called {\em $\aone$-naive} if for every affine $X \in \Sm_k$ the map $\Singaone \mathscr{F}(X) \to \tilde{\mathscr{F}}(X)$ is a weak equivalence of simplicial sets.  As observed in \cite[Remark 2.1.2]{AHWII}, if $\mathscr{F}$ is $\aone$-naive, then for every affine $X \in \Sm_k$ the map
\[
\pi_0(\Singaone \mathscr{F}(X)) \longrightarrow [X,\mathscr{F}]_{\aone}
\]
is a bijection.

By \cite[Proposition 2.1.3]{AHWII}, $\mathscr{F}$ is $\aone$-naive if and only if $\Singaone \mathscr{F}$ satisfies affine Nisnevich excision in the sense of \cite[Section 2.1]{AHW}.  In that case, $R_{\Zar} \Singaone \mathscr{F}$ is Nisnevich-local and $\aone$-invariant.

\begin{thm}
\label{thm:representabilityofNisG}
If $\op{G}$ is an ``isotropic" reductive $k$-group scheme, then $B_{\Nis}\op{G}$ is $\aone$-naive.  In particular, the canonical map
\[
\op{H}^1_{\Nis}(X,\op{G}) \longrightarrow [X,B\op{G}]_{\aone}
\]
is a bijection for every affine $X\in \Sm_k$.
\end{thm}

\begin{proof}
Combine \cite[Theorem 2.3.5]{AHWII} with Theorem~\ref{thm:affinehomotopyinvarianceforGtorsors}.
\end{proof}

Suppose $\op{H} \to \op{G}$ is a closed immersion of ``isotropic" reductive $k$-group schemes.  By appeal to \cite[Th\'eor\`eme 4.C]{Anantharaman}, the quotient $\op{G}/\op{H}$ exists as a smooth $k$-scheme.  Since the map $\op{G} \to \op{G}/\op{H}$ is an $\op{H}$-torsor, it follows that the quotient is smooth since $\op{G}$ has the same property.  That the quotient is affine follows from the fact that $\op{H}$ is reductive and may be realized as $\Spec \Gamma(\op{G},\O_{\op{G}})^{\op{H}}$ (\cite[Theorems 9.1.4 and 9.7.6]{Alper}; for later use, observe that these statements hold over an arbitrary base).  Since $\op{G}$ and $\op{H}$ are reductive, they are connected by assumption, and the connectness statement for the quotient follows.  Granted these fact, we establish Theorem~\ref{thmintro:homogeneousspaces}.

\begin{thm}
\label{thm:representabilityforhomogeneousspaces}
If $\op{H} \to \op{G}$ is a closed immersion of ``isotropic" reductive $k$-group schemes, and if the $\op{H}$-torsor $\op{G} \to \op{G}/\op{H}$ is Nisnevich locally split, then $\op{G}/\op{H}$ is $\aone$-naive.
\end{thm}

\begin{proof}
Combine \cite[Theorem 2.4.2]{AHWII} with Theorem~\ref{thm:affinehomotopyinvarianceforGtorsors}.
\end{proof}

The following result generalizes \cite[Theorem 4]{AHWII}.

\begin{thm}
\label{thm:generalizedflagsareaonenaive}
Assume $\op{G}$ is an ``isotropic" reductive $k$-group scheme and $\op{P} \subset \op{G}$ is a parabolic $k$-subgroup possessing an isotropic Levi factor (e.g., if $\op{G}$ is split), then $\op{G}/\op{P}$ is $\aone$-naive.
\end{thm}

\begin{proof}
Let $\op{L}$ be a Levi factor for $\op{P}$.  The quotients $\op{G}/\op{L}$ and $\op{G}/\op{P}$ exist; see, e.g., \cite[Lemma 3.1.5]{AHWII}.  Moreover, the map $\op{G}/\op{L} \to \op{G}/\op{P}$ induced by the inclusion is a composition of torsors under vector bundles.  Under the assumption that $\op{L}$ is ``isotropic", $\op{G}/\op{L}$ is $\aone$-naive by appeal to Theorem~\ref{thm:representabilityforhomogeneousspaces}.  The fact that $\op{G}/\op{P}$ is $\aone$-naive then follows by appeal to \cite[Lemma 4.2.4]{AHWII} using the fact that $\op{G}/\op{L} \to \op{G}/\op{P}$ is a composition of torsors under vector bundles.
\end{proof}

\subsection{Local triviality of homogeneous spaces}
In order to apply Theorem~\ref{thmintro:homogeneousspaces}, we need a criterion to establish that if $\op{H} \subset \op{G}$ is a group homomorphism, the quotient map $\op{G} \to \op{G}/\op{H}$ is Nisnevich locally trivial.  In this section, we develop some criteria to guarantee this condition holds.

\subsubsection*{Criteria for Nisnevich-local triviality}
\begin{lem}
\label{lem:reductiontorostinv}
Assume $R$ is a commutative unital ring of finite Krull dimension and suppose that $\op{H} \subset \op{G}$ is an inclusion of split reductive $R$-group schemes.
\begin{enumerate}[noitemsep,topsep=1pt]
\item The quotient $\op{G}/\op{H}$ exists as a (connected) smooth affine scheme.
\item The $\op{H}$-torsor $\op{G} \to \op{G}/\op{H}$ is Nisnevich-locally trivial if for any field $K$, the map $\op{H}^1_{\fppf}(K,\op{H}) \to \op{H}^1_{\fppf}(K,\op{G})$ has trivial kernel.
\end{enumerate}
If $R$ is a field, the same results hold without the splitness assumptions.
\end{lem}

\begin{proof}
We first treat the case with splitness assumptions in place.  In that case, split reductive group schemes are pulled back from $\mathbb{Z}$-group schemes.  For both claims, it suffices to prove the result with $R = \Z$: formation of quotients commutes with base-change, affineness and Nisnevich local triviality will be preserved by base-change as well.  Assuming $R = \Z$, the existence of the quotient and the relevant properties are established before the statement of Theorem~\ref{thm:representabilityforhomogeneousspaces}

Now we establish the second statement. To show the relevant torsor is Nisnevich locally trivial, it suffices, by appeal to \cite[Proposition 2]{BB}, to show that the $\op{H}$-torsor in question is rationally trivial, i.e., trivial over the generic point of $\op{G}/\op{H}$ (which is an integral affine $\Z$-scheme).  To that end, the generic point is the spectrum of the fraction field $K$ of the ring $\Gamma(\op{G}/\op{H},\O_{\op{G}/\op{H}})$ and it suffices to show that the restriction of $\op{G} \to \op{G}/\op{H}$ admits a section upon restriction to $K$.  However, the pullback of $\op{G}\to \op{G}/\op{H}$ along the map $\Spec K \to \op{G}/\op{H}$ is an $\op{H}$-torsor on $\Spec K$ whose associated $\op{G}$-torsor is trivial.  The condition that the map $\op{H}^1_{\fppf}(K,\op{H}) \to \op{H}^1_{\fppf}(K,\op{G})$ has trivial kernel precisely guarantees that this  $\op{H}$-torsor over $\Spec K$ is trivial, i.e., admits a section.

When $R$ is a field, one proceeds in an analogous fashion: the existence of and properties of the quotient follow exactly as above. To establish Nisnevich local triviality, one replaces the reference to \cite[Proposition 2]{BB} above with a reference to \cite[Theorem 4.5]{Nisnevich} (note Nisnevich's result is stated for semi-simple groups, but the argument works for reductive group schemes; this is mentioned, e.g., in \cite[\S 1.1]{FedorovPanin}).
\end{proof}

\subsubsection*{The Rost invariant and Nisnevich-local triviality}
Assume $\op{G}$ is a simple, simply-connected algebraic group over a field $F$. The Rost invariant of $\op{G}$ is a natural transformation of functors on the category of field extensions of $F$:
\[
\op{H}^1_{\et}(-,\op{G}) \stackrel{r_{\op{G}}}{\longrightarrow} \op{H}^3_{\et}(-,\Q/\Z(2));
\]
see \cite[Appendix A]{GMS} for more details regarding the group on the right (it will not be important here).  What is important is that the Rost invariant is functorial for homomorphisms of simply connected groups \cite[Proposition 9.4]{GMS}.  In other words, if $\varphi: \op{G}_1 \to \op{G}_2$ is a homomorphism of simply-connected reductive algebraic groups, then there is a commutative diagram of the form
\begin{equation}
\label{eqn:functorialityofRost}
\xymatrix{
\op{H}^1(F,\op{G}_1) \ar[r]^-{r_{\op{G}_1}}\ar[d] & \op{H}^3(F,\Q/\Z(2)) \ar[d]^{n_\varphi} \\
\op{H}^1(F,\op{G}_2) \ar[r]_-{r_{\op{G}_2}} & \op{H}^3(F,\Q/\Z(2))
}
\end{equation}
where $n_\varphi$ is an integer called the {\em Dynkin index} of the homomorphism $\varphi$ or the {\em Rost multiplier} of $\varphi$.

If $\op{G}$ is semi-simple and simply-connected, then an $n$-dimensional $k$-rational representation $\rho$ of $\op{G}$ yields an embedding $\rho: \op{G} \to \op{SL}_n$; we refer to the Dynkin index of this homomorphism as the Dynkin index of the representation.  The Dynkin index then has the following properties:
\begin{enumerate}[noitemsep,topsep=1pt]
\item it is a non-negative integer that is $0$ if and only if the homomorphism is trivial;
\item the Rost multiplier of a composite is the product of the Rost multipliers (\cite[Proposition 7.9]{GMS});
\item if $\rho_1$ and $\rho_2$ are two representations of $\op{G}$, then $n_{\rho_1 \oplus \rho_2} = n_{\rho_1} + n_{\rho_2}$;
\item the Dynkin index of the adjoint representation is the dual Coxeter number.
\end{enumerate}
One then deduces the following criterion for detecting Nisnevich local triviality.

\begin{lem}
\label{lem:rostinvariantcriterionfortriviality}
Assume $\varphi: \op{H} \subset \op{G}$ is a closed immersion group homomorphism of simply-connected semi-simple $k$-group schemes.  If (i) the Dynkin index for $\varphi$ is $1$, and (ii) for every extension $K/k$, the Rost invariant for $\op{H}_K$ is trivial, then the torsor $\op{G} \to \op{G}/\op{H}$ is Nisnevich locally trivial.
\end{lem}

\begin{proof}
By Lemma~\ref{lem:reductiontorostinv}, it suffices to prove: for every extension $K/k$, if given an $\op{H}$-torsor $P$ over $\Spec K$ such that the associated $\op{G}$-torsor $P'$ over $\Spec K$ (obtained by extending the structure group via $\varphi$) is trivial, then $P$ is already trivial.

Assume the kernel of the Rost invariant for $\op{H}$ is trivial for every extension $K/k$.  Suppose $P$ is an $\op{H}_K$-torsor over $\Spec K$, and the associated associated $\op{G}_K$-torsor $P'$ over $\Spec K$ is trivial.  Since the Rost invariant of $P'$ is necessarily trivial, the assumption that $\varphi$ has Rost multiplier $1$ implies $P$ has trivial Rost invariant.  However, since the Rost invariant for $\op{H}_K$ was assumed to be injective, we conclude that $P$ is trivial, which is precisely what we wanted to show.
\end{proof}

For quasi-split groups of low rank, the Rost invariant is frequently injective \cite{Garibaldi}.  Indeed, Garibaldi shows \cite[Theorems 0.1 and 0.5]{Garibaldi} that the Rost invariant is trivial in the following cases:
\begin{enumerate}[noitemsep,topsep=1pt]
\item quasi-split groups of absolute rank $\leq 5$;
\item quasi-split groups of type $\op{B}_6, \op{D}_6$ or $\op{E}_6$;
\item quasi-split groups of type $\op{E}_7$ or split groups of type $\op{D}_7$.
\end{enumerate}
Thus we obtain a number of Nisnevich local triviality results by computation of Dynkin indices.

\begin{ex}
\label{ex:fibersequences}
The Rost multiplier of the inclusion of split groups $\op{Spin}_9 \hookrightarrow \op{F}_4$ is $1$, so Lemma~\ref{lem:rostinvariantcriterionfortriviality} combined with \cite[Theorems 0.1 and 0.5]{Garibaldi} imply that the $\op{Spin}_9$-torsor $\op{F}_4 \to \op{F}_4/\op{Spin}_9$ is Nisnevich locally trivial.  Similar results hold for $\op{F}_4 \subset \op{E}_6$ and $\op{E}_6 \subset \op{E}_7$ (see \cite{Garibaldi} for more details).  Thus, in each of these case, Theorem~\ref{thm:representabilityforhomogeneousspaces} applies and guarantees that the relevant homogeneous space is $\aone$-naive.
\end{ex}

\begin{rem}
Following \cite{AHWOctonion}, one can use the $\aone$-fiber sequences associated with inclusions appearing in Example~\ref{ex:fibersequences} to deduce results about reduction of the corresponding structure groups for (Nisnevich locally trivial) torsors over smooth affine schemes.  Moreover, torsors under the various group schemes above are related to classical algebraic invariants (e.g., $\op{F}_4$-torsors correspond to Albert algebras, $\op{E}_6$- and $\op{E}_7$-torsors correspond to certain structurable algebras \cite{GaribaldiStructurable}).  In light of these applications, we pose the following question, which would be especially interesting to analyze in the cases mentioned in Example~\ref{ex:fibersequences}.
\end{rem}

\begin{question}
\label{question:homogspaceconnectivity}
Suppose $\op{H} \to \op{G}$ is a closed immersion of ``isotropic" reductive $k$-group schemes such that $\op{G} \to \op{G}/\op{H}$ is Nisnevich locally trivial.  
\begin{itemize}[noitemsep,topsep=1pt]
\item What is the $\aone$-connectivity of $\op{G}/\op{H}$?  
\item What is the structure of the first non-vanishing $\aone$-homotopy sheaf of $\op{G}/\op{H}$?
\end{itemize}
\end{question}

\subsubsection*{Motivic spheres as homogeneous spaces}
In \cite{BorelSpheres}, Borel completed the classification of homogeneous spaces that are spheres.  We now establish a similar result for motivic spheres.  To this end, we write $\op{Q}_{2n-1}$ for the split smooth affine quadric defined by the equation $\sum_{i=1}^n x_iy_i = 1$, and $\op{Q}_{2n}$ for the split smooth affine quadric defined by the equation $\sum_{i=1}^n x_iy_i = z(1-z)$.  In \cite[Theorem 2]{ADF}, we showed that $\op{Q}_{2n}$ is $\aone$-weakly equivalent to $S^n \sma \gm{\sma n}$, and it is well known that $\op{Q}_{2n-1}$ is $\aone$-weakly equivalent to $S^{n-1} \sma \gm{\sma n}$.

\begin{thm}
\label{thm:oddspheres}
Suppose $R$ is a commutative base ring.  The following homogeneous spaces are isomorphic to odd-dimensional motivic spheres:
\begin{enumerate}[noitemsep,topsep=1pt]
\item the quotients $\op{SL}_n/\op{SL}_{n-1}$, $\op{SO}_{2n}/\op{SO}_{2n-1}$, and $\op{Sp}_{2m}/\op{Sp}_{2m-1}$ (with $n = 2m$) are isomorphic to $Q_{2n-1}$;
\item the quotient $\op{Spin}_{7}/\op{G}_2$ is isomorphic to $\op{Q}_7$; and
\item the quotient $\op{Spin}_9/\op{Spin}_7$ is isomorphic to  $\op{Q}_{15}$.
\end{enumerate}
Furthermore, for each pair $(\op{G},\op{H})$ as above, the torsor $\op{G} \to \op{G}/\op{H}$ is Zariski locally trivial.
\end{thm}

\begin{proof}
All of these results are presumably well-known.  The first three appear in \cite[\S 4.2]{AHWII} while the last one appears in \cite[Theorem 2.3.5]{AHWOctonion}.  It remains to identify $\op{Spin}_9/\op{Spin}_7 \cong \op{Q}_{15}$; this is essentially classical, so we provide an outline.

We use the notation of \cite[\S 2]{AHWOctonion}.  Let $O$ be the split octonion algebra over $\Z$, and consider the closed subscheme in the scheme $O \times O$ defined by $\op{N}_O(x) - \op{N}_O(y) = 1$ (see \cite[Definition 2.1.9]{AHWOctonion} for an explicit formula for the norm); this scheme is isomorphic to $\op{Q}_{15}$ by definition.  The space $O \times O$ carries the split quadratic form of rank $16$.  However, there is an induced action of $\op{Spin}_9$ on $\op{Q}_{15}$ coming from the spinor representation.

We now repeat the arguments at the beginning of the proof of \cite[Theorem 2.3.5]{AHWOctonion}.  We may first assume without loss of generality that $R = \Z$ and the result in general follows by base-change.  In that case, the relevant quotient exists by appeal to \cite[Th\'eor\`eme 4.C]{Anantharaman}.

The action of $\op{Spin}_9$ on $\op{Q}_{15}$ described above gives a morphism $\op{Spin}_9 \to \op{Q}_{15}$ by choice of a point.  It remains to show that this map induces an isomorphism of quotients.  As in the proof of \cite[Theorem 2.3.5]{AHWOctonion} we may reduce to the case of geometric points.  Having reduced to geometric points, transitivity may be established and the stabilizer identified by a straightforward (and classical) computation using Clifford algebras (see \cite[C.4]{Conrad} for a discussion of the relevant groups).

For Zariski local triviality, it suffices to show that if given a local ring $R$ and $\mathscr{P}$ a $\op{Spin}_7$-torsor over $R$, triviality of the associated $\op{Spin}_9$-torsor implies triviality of $\mathscr{P}$.  Equivalently, if the quadratic space associated with the $\op{Spin}_9$-torsor is split, then the initial quadratic space must also be split; this follows from Witt's cancellation theorem \cite[Theorem 8.4]{EKM}.
\end{proof}

\begin{rem}
Following \cite[Th{\'e}or{\`e}me 3]{BorelSpheres}, it seems reasonable to expect that the list above should be a complete list of homogeneous spaces that are isomorphic to odd-dimensional motivic spheres, at least over an algebraically closed field. 
\end{rem}

\begin{thm}
\label{thm:evenspheres}
If $k$ is a field having characteristic unequal to $2$, then $\op{Q}_{2n}$ is $\aone$-naive.
\end{thm}

\begin{proof}
By \cite[Lemma 3.1.7]{AHWII}, we know that under the hypotheses $\op{Q}_{2n} \cong \op{SO}_{2n+1}/\op{SO}_{2n}$ and that the torsor $\op{SO}_{2n+1} \to \op{SO}_{2n+1}/\op{SO}_{2n}$ is Zariski locally trivial.  Since $\op{SO}_{m}$ is split, the result follows by appeal to Theorem~\ref{thm:representabilityforhomogeneousspaces}.
\end{proof}


\begin{footnotesize}
\bibliographystyle{alpha}
\bibliography{bundlesredux}

\begin{thebibliography}{AHW17b}

\bibitem[ADF17]{ADF}
A.~Asok, B.~Doran, and J.~Fasel.
\newblock Smooth models of motivic spheres and the clutching construction.
\newblock {\em Int. Math. Res. Not. IMRN}, (6):1890--1925, 2017.

\bibitem[AHW17a]{AHW}
A.~Asok, M.~Hoyois, and M.~Wendt.
\newblock Affine representability results in {$\Bbb A^1$}-homotopy theory, {I}:
  vector bundles.
\newblock {\em Duke Math. J.}, 166(10):1923--1953, 2017.

\bibitem[AHW17b]{AHWOctonion}
A.~Asok, M.~Hoyois, and M.~Wendt.
\newblock Generically split octonion algebras and {${\mathbb A}^1$}-homotopy
  theory.
\newblock 2017.
\newblock Available at \url{https://arxiv.org/abs/1704.03657}.

\bibitem[AHW18]{AHWII}
A.~Asok, M.~Hoyois, and M.~Wendt.
\newblock Affine representability results in {$\Bbb A^1$}-homotopy theory,
  {II}: {P}rincipal bundles and homogeneous spaces.
\newblock {\em Geom. Topol.}, 22(2):1181--1225, 2018.

\bibitem[Alp14]{Alper}
J.~Alper.
\newblock Adequate moduli spaces and geometrically reductive group schemes.
\newblock {\em Algebr. Geom.}, 1(4):489--531, 2014.

\bibitem[Ana73]{Anantharaman}
S.~Anantharaman.
\newblock Sch\'emas en groupes, espaces homog\`enes et espaces alg\'ebriques
  sur une base de dimension 1.
\newblock pages 5--79. Bull. Soc. Math. France, M\'em. 33, 1973.

\bibitem[BB70]{BB}
A.~Bia{\l}ynicki-Birula.
\newblock Rationally trivial homogeneous principal fibrations of schemes.
\newblock {\em Invent. Math.}, 11:259--262, 1970.

\bibitem[Bor50]{BorelSpheres}
A.~Borel.
\newblock Le plan projectif des octaves et les sph\`eres comme espaces
  homog\`enes.
\newblock {\em C. R. Acad. Sci. Paris}, 230:1378--1380, 1950.

\bibitem[Con14]{Conrad}
B.~Conrad.
\newblock Reductive group schemes.
\newblock In {\em Autour des sch\'emas en groupes. {V}ol. {I}}, volume 42/43 of
  {\em Panor. Synth\`eses}, pages 93--444. Soc. Math. France, Paris, 2014.

\bibitem[CTHK97]{CTHK}
J.-L. Colliot-Th\'el\`ene, R.~T. Hoobler, and B.~Kahn.
\newblock The {B}loch-{O}gus-{G}abber theorem.
\newblock In {\em Algebraic {$K$}-theory ({T}oronto, {ON}, 1996)}, volume~16 of
  {\em Fields Inst. Commun.}, pages 31--94. Amer. Math. Soc., Providence, RI,
  1997.

\bibitem[CTO92]{CTO}
J.-L. Colliot-Th\'el\`ene and M.~Ojanguren.
\newblock Espaces principaux homog\`enes localement triviaux.
\newblock {\em Inst. Hautes \'Etudes Sci. Publ. Math.}, (75):97--122, 1992.

\bibitem[DG70]{SGA3-3}
Michel Demazure and Alexandre Grothendieck.
\newblock {\em Structure des {S}ch\'emas en {G}roupes {R}\'eductifs}, volume
  153 of {\em Lecture Notes in Mathematics}.
\newblock Springer, 1970.

\bibitem[EKM08]{EKM}
R.~Elman, N.~Karpenko, and A.~Merkurjev.
\newblock {\em The algebraic and geometric theory of quadratic forms},
  volume~56 of {\em American Mathematical Society Colloquium Publications}.
\newblock American Mathematical Society, Providence, RI, 2008.

\bibitem[FP15]{FedorovPanin}
R.~Fedorov and I.~Panin.
\newblock A proof of the {G}rothendieck-{S}erre conjecture on principal bundles
  over regular local rings containing infinite fields.
\newblock {\em Publ. Math. Inst. Hautes \'Etudes Sci.}, 122:169--193, 2015.

\bibitem[Gar01a]{Garibaldi}
R.~S. Garibaldi.
\newblock The {R}ost invariant has trivial kernel for quasi-split groups of low
  rank.
\newblock {\em Comment. Math. Helv.}, 76(4):684--711, 2001.

\bibitem[Gar01b]{GaribaldiStructurable}
R.~S. Garibaldi.
\newblock Structurable algebras and groups of type {$E_6$} and {$E_7$}.
\newblock {\em J. Algebra}, 236(2):651--691, 2001.

\bibitem[GMS03]{GMS}
S.~Garibaldi, A.~Merkurjev, and J.-P. Serre.
\newblock {\em Cohomological invariants in {G}alois cohomology}, volume~28 of
  {\em University Lecture Series}.
\newblock American Mathematical Society, Providence, RI, 2003.

\bibitem[HK]{HK}
A.~Hogadi and G.~Kulkarni.
\newblock Gabber's presentation lemma for finite fields.
\newblock {\em {\em To appear} J. Reine Angew. Math.}
\newblock \url{http://dx.doi.org/10.1515/crelle-2017-0049}.

\bibitem[Lan56]{Lang}
S.~Lang.
\newblock Algebraic groups over finite fields.
\newblock {\em Amer. J. Math.}, 78:555--563, 1956.

\bibitem[Mor12]{MField}
F.~Morel.
\newblock {\em {$\mathbb A^1$}-algebraic topology over a field}, volume 2052 of
  {\em Lecture Notes in Mathematics}.
\newblock Springer, Heidelberg, 2012.

\bibitem[MV99]{MV}
F.~Morel and V.~Voevodsky.
\newblock {${\mathbb A}^1$}-homotopy theory of schemes.
\newblock {\em Inst. Hautes \'Etudes Sci. Publ. Math.}, (90):45--143 (2001),
  1999.

\bibitem[Nis84]{Nisnevich}
Y.~A. Nisnevich.
\newblock Espaces homog\`enes principaux rationnellement triviaux et
  arithm\'etique des sch\'emas en groupes r\'eductifs sur les anneaux de
  {D}edekind.
\newblock {\em C. R. Acad. Sci. Paris S\'er. I Math.}, 299(1):5--8, 1984.

\bibitem[Rag78]{Raghunathan1}
M.~S. Raghunathan.
\newblock Principal bundles on affine space.
\newblock In {\em C. {P}. {R}amanujam---a tribute}, volume~8 of {\em Tata Inst.
  Fund. Res. Studies in Math.}, pages 187--206. Springer, Berlin-New York,
  1978.

\bibitem[Rag89]{Raghunathan2}
M.~S. Raghunathan.
\newblock Principal bundles on affine space and bundles on the projective line.
\newblock {\em Math. Ann.}, 285(2):309--332, 1989.

\end{thebibliography}
\end{footnotesize}
\Addresses
\end{document}